\documentclass[a4paper,11pt]{article}

\usepackage{amsmath,amsthm,amssymb,fullpage}
\usepackage{enumitem}
\setlist[enumerate,itemize]{noitemsep}
\setlist[enumerate,1]{label=(\roman*)}

\newtheorem{theorem}{Theorem}

\title{Lonely passengers: a short proof}
\author{John Haslegrave}

\begin{document}
	\maketitle
\begin{abstract}
	A fixed number of passengers independently board one of several buses uniformly at random. The lonely passenger problem is to prove that the probability of at least one passenger being the only one in their bus is increasing in the number of buses. It was solved in a strong form by Imre P\'eter T\'oth, who proved stochastic dominance of the number of such passengers as the number of buses increases, but observed that, surprisingly, no short proof was known ``despite the efforts of several experts''. We give a very short proof of the weaker result. The proof of the strong form, using the same idea, is more involved but still relatively short.
\end{abstract}

\section{Introduction}
	Suppose $n\geq 2$ passengers board $k\geq 1$ buses, with each passenger independently choosing a bus uniformly at random. We say that a passenger is \textit{lonely} if no other passengers board the same bus. If we keep $n$ fixed, but increase $k$, does this make it (strictly) more likely that at least one passenger is lonely?
	
	Intuitively we might expect the answer to be ``yes'', and indeed it is trivial that the expected number of lonely passengers is strictly increasing in the number of buses. However, this problem, originally posed by Imre P\'eter T\'oth in September 2023, to illustrate a key difficulty in a conjecture on random walks, remained unsolved for some time, and the proof eventually obtained in 2025 \cite{Toth} is surprisingly long and intricate. T\'oth remarked that the problem had resisted the efforts of several experts in probability, with multiple incorrect solutions put forward, and that it was surprising that no short proof of this deceptively simple statement had been found.
	
	In this paper we present a very short proof of the problem as described above. We then adapt the method to prove the stronger version actually given by T\'oth's solution, that there is stochastic dominance in the number of lonely passengers, i.e.\ $L^{(k)}_n\prec L^{(k+1)}_n$, where the random variable $L^{(k)}_n$ is the number of lonely passengers when $n$ passengers board $k$ buses. We write $p_{n,k,r}$ for $\mathbb P(L^{(k)}_n\geq r)$; strict stochastic dominance requires that $p_{n,k,r}\leq p_{n,k+1,r}$ for each $r$ with strict inequality in at least one case (which will be the case $r=1$).
	
	We remark, following \cite{Toth}, that this problem arose from a conjecture of L\'aszl\'o M\'arton T\'oth, which remains open. Consider a simple symmetric random walk of length $n$ on an infinite tree $T$, having all degrees at least $3$. A \textit{regeneration} for this walk is an edge that is crossed exactly once. T\'oth conjectured that the probability of a regeneration is minimised when $T$ is the $3$-regular tree, or more generally that the number of regenerations is stochastically minimal in the same case. The lonely passenger problem was developed in order to highlight a specific obstacle in analysing this problem.
	This conjecture would have consequences for the theory of Ramanujan graphings, introduced by Backhausz, Szegedy and Vir\'ag \cite{BSV} as a graph limit analogue of Ramanujan graphs (see \cite{Lov} for background on graph limits). In particular, it would imply that Bernoulli graphings of unimodular Galton--Watson trees are Ramanujan. (See \cite[Section 1]{AL07} for the construction of unimodular Galton--Watson trees.)

\section{Proofs}
	We first give the short proof that the probability of a lonely passenger is strictly increasing in the number of buses. The basic idea is that we can simulate the $k$-bus case by starting with $k+1$ buses, and then taking bus $k+1$ out of service. We then directly compare the probability that there is a lonely passenger before reassignment but not afterwards with the reverse.
	\begin{theorem}\label{first}For each $n\geq 2$ and $k\geq 1$ we have $p_{n,k+1,1}>p_{n,k,1}$.
	\end{theorem}
	\begin{proof}
	First assign passengers independently and uniformly to $k+1$ buses, and generate a list $Y_1,\ldots, Y_n$ of i.i.d.\ random variables uniformly distributed on $[k]$, independently of the assignment to buses. We say that a \textit{configuration} is an element of $[k+1]^n\times [k]^n$, giving the buses to which passengers $1$ to $n$ are initially assigned, followed by the values of $Y_1$ to $Y_n$. Note that all configurations are equally likely. Then reassign any passengers on bus $k+1$ to buses $Y_1,\ldots, Y_m$, where $m$ is the number of passengers needing reassignment. Note that the allocation after reassignment follows the law for $k$ buses.
	
	Let $X_1,\ldots,X_{k+1}$ be the number of passengers initially assigned to each bus, and let $X'_1,\ldots,\allowbreak X'_{k}$ be the number after reassignment. Consider the events $A=\{\exists i\in[k+1]:X_i=1\}$ and $A'=\{\exists i\in[k]:X'_i=1\}$; clearly it suffices to show that $\mathbb P(A\cap A'^{\mathsf c})>\mathbb P(A'\cap A^{\mathsf c})$.
	
	For each $m\in\{2,\ldots,n\}$ let $B_m$ be the event that $X_{k+1}=1$ and $X_{Y_1}=m-1$, but $X_i\neq 1$ for all $i\neq Y_1,k+1$. We have $B_m\subseteq A\cap A'^{\mathsf c}$, since $X_{k+1}=1$ but $X'_i=X_i\neq 1$ if $i\neq Y_1, k+1$ and $X'_{Y_1}=m>1$. The $B_m$ are disjoint, so $\mathbb P(A\cap A'^{\mathsf c})\geq\sum_{m=2}^n\mathbb P(B_m)$.
	
	Similarly, for $m\in\{2,\ldots,n\}$ let $B'_m$ be the event that $X_{k+1}=m$ and $X_i\neq 1$ for all $i\leq k$, and $X_{Y_j}=0$ for at least one $j\leq m$. Note that $B'_m\supseteq A'\cap A^{\mathsf c}\cap\{X_{k+1}=m\}$, and this inclusion is strict for $m=n$, since if $X_{k+1}=n$ and the $Y_i$ are all equal then $B'_n$ holds but $A'$ doesn't. Consequently (noting that $A'\cap A^{\mathsf c}\cap\{X_{k+1}\leq 1\}=\varnothing$) we have $\mathbb P(A'\cap A^{\mathsf c})< \sum_{m=2}^n\mathbb P(B'_m)$.
	
	Let $C_{m,i}$ for $i\leq m$ be the event $B'_m\cap\{X_{Y_i}=0\}$. Clearly $B'_m=\bigcup_{i=1}^m C_{m,i}$, and $\mathbb P(C_{m,i})$ does not depend on $i$, giving $\mathbb P(B'_m)\leq m\mathbb P(C_{m,1})$. 
	To complete the proof, we observe that $\mathbb P(B_m)= m\mathbb P(C_{m,1})$. Indeed, there is a bijection between configurations where $B_m$ occurs and pairs $\sigma, i$ where $\sigma$ is a configuration where $C_{m,1}$ occurs, and $i$ is one of the $m$ passengers assigned to bus $k+1$ in $\sigma$. In one direction, move all passengers from bus $Y_1$ to bus $k+1$; in the other, move all but passenger $i$ from bus $k+1$ to bus $Y_1$.
	\end{proof}
	Next we develop the same method to prove stochastic dominance. 
	\begin{theorem}
		For each $n\geq 2$ and $k\geq 1$ we have $L_n^{(k+1)}\succ L_n{(k)}$.
	\end{theorem}
	\begin{proof}
		In view of Theorem \ref{first} it is sufficient to prove $p_{n,k+1,r}\geq p_{n,k,r}$ for every $n\geq r\geq 2$ and $k\geq 1$.
		We use the same setup of configurations and reassignment as before. We wish to show that $\mathbb P(A_r\cap A_r'^{\mathsf c})\geq\mathbb P(A_r'\cap A_r^{\mathsf c})$, where $A_r$ is the event that at least $r$ of $X_1,\ldots, X_{k+1}$ take value $1$, and $A'_r$ is the corresponding event for $X'_1,\ldots,X'_k$.
		
		For each $m\geq 2$ and $\ell\geq 0$ with $0<r-\ell\leq m$, let $D'_{m,\ell}$ be the event that all the following hold:
		\begin{itemize}
			\item $X_{k+1}=m$;
			\item there are exactly $\ell$ choices of $i\leq k$ such that $X_i=1$;
			\item there are at least $r-\ell$ values of $s\leq k$ such that $X_s=0$ and there is a unique $j\leq m$ with $Y_j=s$; and
			\item if there are exactly $r-\ell$ such $s$, then there is no $j\leq m$ with $X_{Y_j}=1$.
		\end{itemize}
		Note that these events are disjoint. We claim their union contains $A'_r\cap A_r^{\mathsf c}$. Indeed, for any configuration $\sigma$ in $A'_r\cap A_r^{\mathsf c}$, we must have $X_{k+1}=m$ for some $m\geq 2$, the number of $i\leq k$ with $X_i=1$ must be $\ell$ for some $\ell<r$, and there must be at least $r-\ell$ choices of $s\leq k$ such that $X_s=0$ but $X'_s=1$. Each such $s$ satisfies $Y_j=s$ for exactly one $j\leq m$, since the number of such $j$ is precisely $X'_s-X_s$. Furthermore, if there are exactly $r-\ell$ such $s$, then for each $i\leq k$ with $X_i=1$ we also have $X'_i=1$, i.e.\ there is no $j\leq m$ with $Y_j=i$. Thus $\sigma\in D'_{m,\ell}$. (However, configurations in $D'_{m,\ell}$ might not be in $A'$, since we do not forbid configurations that gain a lonely passenger in $r-\ell+1$ buses but lose one in $2$ buses.) Consequently, writing $P=\{(m,\ell)\in\mathbb Z^2:m\geq 2\text{ and }\ell\geq 0\text{ and }0<r-\ell\leq m\}$, we have 
		\begin{equation}\mathbb P(A'_r\cap A_r^{\mathsf c})\leq\sum_{(m,\ell)\in P}\mathbb P(D'_{m,\ell}).\label{eq:a-dprime}\end{equation}
		
		For each $S\in\binom{[m]}{r-\ell}$ (where $(m,\ell)\in P$ is fixed), we define $E_{m,\ell,S}$ to be the event that $D'_{m,\ell}$ occurs with the elements of $S$ satisfying the third point above. Since each $E_{m,\ell,S}$ is equiprobable, and their union is $D'_{m,\ell}$, we have 
		\begin{equation}\mathbb P(D'_{m,\ell})\leq\binom{m}{r-\ell}\mathbb{P}(E_{m,\ell,[r-\ell]}).\label{eq:dprime-e}\end{equation}
		
		Next we define corresponding events $D_{m,\ell}\subseteq A_r\cap A'^{\mathsf c}_r$. For a configuration $\sigma$, let $J(\sigma)$ be the minimum $j\geq 1$ such that $X_{Y_{J(\sigma)}}\neq 1$, if it exists, and set $J(\sigma)=0$ otherwise. Let $D$ be the set of configurations $\sigma$ with the following properties:
		\begin{enumerate}
			\item $X_{k+1}=1$;
			\item the number of $i\leq k$ with $X_i=1$ is either $r-1$ or $r$; 
			\item $J(\sigma)>1$; and
			\item if $j,j'< J(\sigma)$ with $j\neq j'$ then $Y_{j}\neq Y_{j'}$.
		\end{enumerate}
		Note that any such $\sigma$ is in $A_r$ by (i) and (ii), but since (iii) gives $X_{Y_1}=1$, we have $X'_{Y_1}=2$ but $X'_i=X_i$ for all other $i\leq k$, which together with (ii) gives $\sigma\not\in A'_r$. Thus $D\subseteq A_r\cap A'^{\mathsf c}_r$.
		
		Define $D_{m,\ell}$ to be the set of $\sigma\in D$ such that one of the following holds:
		\begin{itemize}
			\item the number of $i\leq k$ with $X_i=1$ is $r-1$ and $J(\sigma)+X_{Y_{J(\sigma)}}=m$ and $r-J(\sigma)=\ell$; or
			\item the number of $i\leq k$ with $X_i=1$ is $r$ and $J(\sigma)=m$ and $r-J(\sigma)+1=\ell$.
		\end{itemize}
		Clearly each $\sigma\in D$ uniquely determines integers $m,\ell$ with $\sigma \in D_{m,\ell}$. Note that (iii) implies $m\geq 2$ and (ii), (iii) and (iv) together give $2\leq J(\sigma)\leq r$ in the first case and $2\leq J(\sigma)\leq r+1$ in the second, which implies $\ell\geq 0$ and $0<r-\ell\leq m$.
		Thus the events $(D_{m,\ell})_{(m,\ell)\in P}$ partition $D$, giving
		\begin{equation}\mathbb P(A_r\cap A'^{\mathsf c}_r)\geq\mathbb P(D)=\sum_{(m,\ell)\in P}\mathbb P(D_{m,\ell}).\label{eq:a-d}\end{equation}
		Finally we claim that, for each $(m,\ell)\in P$,
		\[\mathbb P(D_{m,\ell})\geq\frac{m!}{(m+\ell-r)!}\mathbb P(E_{m,\ell,[r-\ell]}),\]
		which together with \eqref{eq:a-dprime}, \eqref{eq:dprime-e} and \eqref{eq:a-d} (and $\frac{m!}{(m+\ell-r)!}=(r-\ell)!\binom{m}{r-\ell}$) completes the proof.
		
		To prove the claim we first associate each configuration $\sigma\in E_{m,\ell,[r-\ell]}$ with the $\frac{m!}{(m+\ell-r)!}$ configurations that can be obtained in the following way: of the $m$ passengers on bus $k+1$ in $\sigma$, leave one there, put one on each bus $Y_i$ for $i\leq r-\ell-1$, and put the remaining $m+\ell-r$ on bus $Y_{r-\ell}$. Let $\tau$ be some configuration obtained from $\sigma$ in this way. Suppose $m+\ell-r\neq 1$. Then, for $\tau$, we have $X_{k+1}=1$, the number of $i\leq k$ with $X_i=1$ is $(r-\ell-1)+\ell=r-1$, and $J(\tau)=r-\ell>1$, and for each $j<J(\sigma)$ there is no other $j'\leq m$ with $Y_{j'}=Y_j$ by definition of $E_{m,\ell,[r-\ell]}$. Furthermore, $J(\tau)+X_{Y_{J(\tau)}}=m$ and $r-J(\tau)=\ell$, so $\tau\in D_{m,\ell}$. Alternatively, suppose $m+\ell-r= 1$. Then, for $\tau$, we have $X_{k+1}=1$, $X_{Y_1}=\cdots=X_{Y_{r-\ell}}=1$, and the number of $i\leq k$ with $X_i=1$ is $(r-\ell)+\ell=r$. Since $r-\ell+1=m$, we obtain by definition of $E_{m,\ell,[r-\ell]}$ that $Y_{r-\ell+1}$ is different from $Y_1,\ldots,Y_{r-\ell}$, and so bus $Y_{r-\ell+1}$ has the same number of passengers in $\tau$ as in $\sigma$. That number cannot be $1$, since this would contradict the definition of $E_{m,\ell,[r-\ell]}$. Thus $J(\tau)=r-\ell+1=m$. It follows that $Y_1,\ldots,Y_{J(\tau)-1}$ are all different, and so $\tau\in D_{J(\tau),r-J(\tau)+1}=D_{m,\ell}$.
		
		It is now immediate that each $\tau\in D_{m,\ell}$ is associated with at most one $\sigma\in E_{m,\ell,[r-\ell]}$, since the only possible configuration from which $\tau$ may be obtained in the above manner is that given by moving all passengers on buses $Y_1,\ldots,Y_{r-\ell}$ (in $\tau$) to bus $k+1$. (Note, however, that this might fail to be in $E_{m,\ell,[r-\ell]}$, since some $Y_j$ with $j\leq r-\ell$ might coincide with some $j'$ with $r-\ell<j'\leq m$.) This proves the claim, completing the proof.
	\end{proof}


\begin{thebibliography}{9}
	\bibitem{AL07}  David Aldous and Russell Lyons, 
	\newblock Processes on unimodular random networks.
	\newblock\textit{Electron. J. Probab.} \textbf{12} (2007), no. 54, 1454--1508.
	\bibitem{BSV} \'Agnes Backhausz, Bal\'azs Szegedy, and B\'alint Vir\'ag, 
	\newblock Ramanujan graphings and correlation decay in local algorithms. 
	\newblock\textit{Random Structures Algorithms} \textbf{47} (2015), no. 3, 424--435.
	\bibitem{Lov} L\'aszl\'o Lov\'asz, 
	\newblock \textit{Large networks and graph limits}. 
	\newblock American Mathematical Society Colloquium Publications, vol. \textbf{60} (2012).
	\bibitem{Toth}Imre P\'eter T\'oth,
	\newblock Lonely passenger problem: the more buses there are, the more lonely passengers there will be.
	\newblock Preprint, arXiv:2501.05758 (2025).
\end{thebibliography}
\end{document}